\documentclass[11pt, a4paper,leqno]{amsart}
\usepackage{amsmath,amsthm,amscd,amssymb,amsfonts,amsbsy,mathrsfs}
\usepackage{latexsym}
\usepackage{txfonts}
\usepackage{exscale}
\usepackage{pdfsync}

\usepackage{enumerate}
\usepackage{amssymb,esint,mathrsfs}
\usepackage{bm}

\usepackage{lipsum}

\usepackage{bbm}
\usepackage[colorlinks,allcolors=RoyalBlue]{hyperref}
\usepackage{pdfsync}
\usepackage[usenames, dvipsnames]{color}
\usepackage[usenames, dvipsnames]{xcolor}

\newcommand{\strt}[1]{\rule{0pt}{#1}}
\newcommand{\car}[1]{\chi_{\strt{1.5ex}#1}}

\makeatletter

\newcommand{\Rmnum}[1]{\expandafter\@slowromancap\romannumeral #1@}
\makeatother

\calclayout
\allowdisplaybreaks

\theoremstyle{plain}
\newtheorem{theorem}{Theorem}
\newtheorem{lemma}[theorem]{Lemma}

\newtheorem{proposition}[theorem]{Proposition}
\newtheorem{conjecture}[theorem]{Conjecture}

\theoremstyle{definition}
\newtheorem{definition}[equation]{Definition}

\newtheorem{example}[equation]{Example}
\newtheorem*{definition*}{Definition}

\theoremstyle{remark}
\newtheorem{remark}[theorem]{Remark}

\numberwithin{equation}{section}
\numberwithin{theorem}{section}

\newcommand{\R}{\mathbb{R}}

\newcommand{\Z}{\mathbb{Z}}

\newcommand{\eps}{\varepsilon}

\newcommand{\la}{\lambda}

\newcommand{\FI}{\varphi}

\newcommand{\norm}[1]{\left\lVert #1\right\rVert}

\DeclareMathOperator{\dist}{dist}
\DeclareMathOperator{\diam}{diam}

\def\Xint#1{\mathchoice
  {\XXint\displaystyle\textstyle{#1}}%
  {\XXint\textstyle\scriptstyle{#1}}%
  {\XXint\scriptstyle\scriptscriptstyle{#1}}%
  {\XXint\scriptscriptstyle\scriptscriptstyle{#1}}%
  \!\int}
\def\XXint#1#2#3{{\setbox0=\hbox{$#1{#2#3}{\int}$}
    \vcenter{\hbox{$#2#3$}}\kern-.5\wd0}}

\def\avgint{\Xint-}
\newcommand{\vertiii}[1]{{\left\vert\kern-0.25ex\left\vert\kern-0.25ex\left\vert #1 
    \right\vert\kern-0.25ex\right\vert\kern-0.25ex\right\vert}}

\author{Javier Canto}
\address{BCAM -- Basque Center for Applied Mathematics, 48009 Bilbao, Spain} 
\email{jcanto@bcamath.org}

\thanks{This research is supported by the Spanish Ministry of Economy and Competitiveness, MTM2017-
82160-C2-2-P and SEV-2017-0718 and by the Basque Government BERC 2018-2021 and IT1247-19. The author is also supported by the Basque Government through the predoctoral grant ``Ayuda para la formaci\'on de personal investigador no doctor".}



\allowdisplaybreaks

\begin{document}

\title[Sharp RHI for $C_p$ weights and Applications]{Sharp reverse H\"older inequality for $C_p$ weights and applications}

\begin{abstract}

We prove an appropriate sharp quantitative reverse H\"older inequality for the $C_p$ class of weights from which we obtain as a limiting case the sharp reverse H\"older inequality for the $A_\infty$ class of weights \cite{HP1,HPR1}. We use this result to provide a quantitative weighted norm inequality between Calder\'on--Zygmund operators and the Hardy--Littlewood maximal function, precisely
\[
\norm{Tf}_{L^p(w)} \lesssim_{T,n,p,q} [w]_{C_q}(1+\log^+[w]_{C_q})\norm{Mf}_{L^p(w)},
\] 
for $w\in C_q$ and $q>p>1$, quantifying Sawyer's theorem \cite{Sawyer1}.
\end{abstract}

\maketitle
\setcounter{tocdepth}{1}
\setcounter{tocdepth}{2}

\section{Introduction and main results}

One of the many forms of the classical H\"older inequality is the following one. For any non-negative function $f$ and $\delta>0$, we have
\begin{equation}\label{eq:HI-classical}
\frac{1}{|Q|}\int_Q f(x)dx \leq \left( \frac{1}{|Q|}\int_Q f(x)^{1+\delta}dx\right)^\frac{1}{1+\delta},
\end{equation}
where $Q\subset \R^n$ is a cube and $|\cdot|$ denotes the Lebesgue measure. Reverse H\"older inequalities (RHI) are the same as \eqref{eq:HI-classical} but with the inequality in the opposite direction. More precisely,
\begin{equation}\label{eq:RHI-class}
 \left( \frac{1}{|Q|}\int_Q f(x)^{r}dx\right)^\frac{1}{r} \leq \frac{C}{|Q|}\int_Q f(x)dx,
\end{equation}
for some $r>1$. There has to be some constant $C\geq 1$, since otherwise it would be an equality. Weights satisfying \eqref{eq:RHI-class} with uniform $C$ for all cubes belong to the class $RH_r$. 

The characterization of weights satisfying a RHI is a classical result: a weight satisfies a RHI if and only if it is contained in the class $A_\infty$. In other words,
\[ A_\infty = \bigcup_{r>1} RH_r.\]

A sharp quantitative RHI for $A_\infty$ was given by Hyt\"onen, P\'erez and Rela \cite{HP1,HPR1}, see Section 2 for details. This quantification has two main properties: the constant on the right hand side is uniform for all weights and the dependence on the $A_\infty$ constant of the exponent is sharp. These sharp RHI inequalities have been used in different applications: for obtaining quantitative estimates for norms of some singular integral operators  \cite{CACDPO,HRT,LPRR}, or for obtaining sharp estimates for solutions of certain PDE \cite{JLX}, among many others.

The main aim of this article is to, mimicking the $A_\infty$ RHI of Hyt\"onen--P\'erez--Rela, give a sharp RHI in the context the $C_p$ class of weights. This class was introduced by Muckenhoupt in \cite{Muckenhoupt1981} and it is intrinsically related to the Coifman--Fefferman inequality (CFI). In our context, CFI is a weighted norm inequality between a Calder\'on--Zygmund operator and the Hardy--Littlewood maximal function. More precisely,
\begin{equation}\label{eq:CF-p}\tag{CFI$-p$}
\int_{\R^n} (T^*f(x))^pw(x)dx \leq c\int_{\R^n} (Mf(x))^pw(x)dx.
\end{equation}
See Section 5 for the precise definitions and for an exposition on the inequality. Recently, inequality \eqref{eq:CF-p} has been shown to hold for a wider variety of operators \cite{CLRT,CLPR}.

This inequality was first proved by Coifman and Fefferman in \cite{CoifmanFeffermanpaper} for $A_\infty$ weights, but Muckenhoupt showed in \cite{Muckenhoupt1981} that $A_\infty$ is not a necessary condition for \eqref{eq:CF-p}. In that article, he gave a necessary condition which he named the $C_p$ condition. Note how the class depends on the exponent $p$. Later on, Sawyer \cite{Sawyer1} proved that $w\in C_{p+\eta}$ for some $\eta>0$ is a sufficient condition in the range $p\in(1,\infty)$. It is still an open conjecture if $C_p$ is a sufficient condition.

Since the $C_p$ class is strictly bigger than $A_\infty$, one cannot expect a true RHI for these weights. Nevertheless, there is a weaker RHI for these weights. Indeed, for $1<p<\infty$, a weight $w$ belongs to $C_p$ if and only if there exist $\delta,C>0$ such that
\begin{equation*}
\left( \avgint_Q w(x)^{1+\delta}dx\right)^\frac1{1+\delta} \leq \frac{C}{|Q|}\int_{\R^n} \big( M\car Q\big(x))^p w(x)dx
\end{equation*}
for every cube $Q$, where $M\car Q$ denotes the Hardy--Littlewood maximal function of the characteristic function of the cube $Q$. Since $M\car Q\geq \car Q$ a.e., this is  weaker than \eqref{eq:RHI-class}. Abusing slightly the language, we shall also call this weaker reverse H\"older inequality a reverse H\"older inequality.

As stated before, the aim of this article is to give a quantitative RHI for $C_p$ weights, with a sharp dependence of the exponent on the weight. In order to do that we define the $C_p$ constant of a weight $w$ as 
\[
[w]_{C_p} := \sup_Q \frac{1}{\int_{\R^n} (M\car Q)^p w} \int_Q M(w \ \car Q).
\]
See Section 2 for the motivation behind this definition. 

\begin{theorem}[Sharp quantitative RHI for $C_p$ weights]
Let $1<p<\infty$ and let $w$ be a weight such that $0<[w]_{C_p}<\infty.$
Then $w\in C_p$ and $w$ satisfies, for $\delta = \frac{1}{ B_{n,p} \max([w]_{C_p},1)}$, 
\begin{equation*}
\left(\avgint_Qw^{1+\delta}\right)^\frac1{1+\delta} \leq  \frac{4}{|Q|}\int_{\R^n} (M\chi_{\strt{1.5ex}Q})^pw.
\end{equation*} 
\end{theorem}
We emphasize that, even though the result is very similar to the sharp $A_\infty$ RHI, the proof is by completely different methods.

Taking advantage of the connection between the classes $A_\infty$ and $C_p$, we are able to obtain the sharp RHI for $A_\infty$ weights as a consequence of the RHI for $C_p$ weights. In this way, we know that the dependence of the $C_p$ constant is sharp.

As it is intrinsically related to the $C_p$ class, the last part of this article is devoted to the CFI. We give a quantification on the weighted inequalities between the Hardy--Littlewood maximal operator and Calder\'on--Zygmund operators. See Section 5 for precise definitions. 

\begin{theorem}
Let $T$ be a Calder\'on--Zygmund operator  and let $q>p>1$. Then, if  $w\in C_q$ and $f \in C_c^\infty(\R^n)$, then the following estimate holds
\begin{equation}\label{eq:CFI-Cp-intro}
\|{T^*f}_{L^p(w)}\|\leq  c_{n,T,p,q} ( [w]_{C_q}+1)\log(e+[w]_{C_q}) \norm{ M  f}_{L^p(w)}   .
\end{equation}
\end{theorem}
The proof of this theorem follows the original article of Sawyer \cite{Sawyer1} with some variants. In particular, the  quantitative RHI for $C_p$ weights above and the use of the good-$\lambda$ inequality with exponential decay of Buckley \cite{Buckley1993} rather than the linear decay of Coifman-Fefferman \cite{CoifmanFeffermanpaper} will play a main role in the argument.

For $A_\infty$ weights, the following quantification of the CFI is known:
\begin{equation}\label{eq:CFI-Ainfty-quant}
\|T^*f\|_{L^p(w)}\leq c_p [w]_{A_\infty}\ \|Mf\|_{L^p(w)}.
\end{equation}
We note that the logarithm on \eqref{eq:CFI-Cp-intro} appears as a consequence of the non-local nature of the $C_p$ condition, but based on \eqref{eq:CFI-Ainfty-quant} and the discussion on Section 4, we conjecture that the correct dependence should be linear:

\begin{conjecture}
Let $T$ and $q,p$ as in the theorem. Then 
$$\norm{T^*f}_{L^p(w)}\leq  c_{n,T,p,q} ( [w]_{C_q}+1) \norm{ M  f}_{L^p(w)}.$$
\end{conjecture}

\section{Preliminaries}

We start by fixing the basic notation. By a weight we mean a non-negative locally integrable function in $\R^n$. Weights will be denoted by the symbol $w$. For a measurable set $E$, $\chi_E$ denotes the characteristic function of $E$. $M$ will denote the Hardy--Littlewood maximal operator $$Mf(x) := \sup_Q \ \frac{\chi_{\strt{1.5ex}Q}(x)}{|Q|} \int_Q|f| ,$$ where the supremum is taken over all cubes with sides parallel to the coordinate axes. For a weight $w$ and a measurable set $E$, $w(E)$ denotes $\int_Ew(x)dx$. 
Also we will be using the notation,  $\avgint_E w= \frac1{|E|} \avgint_E w$ when $E$ is of finite measure. 

We present the definition of $C_p$ as given in \cite{Muckenhoupt1981} and \cite{Sawyer1}.

\begin{definition}[$C_p$ weights] Let $1<p<\infty$. We say that a weight $w$ is of class $C_p$, and we write $w\in C_p$, if there exist $C,\eps>0$ such that for every cube $Q$ and every measurable $E\subset Q$ we have
\begin{equation}\label{Cp-definition}
w(E) \leq C \left(\frac{|E|}{|Q|}\right)^\eps \int_{\R^n} (M\chi_{\strt{1.5ex}Q}(x))^pw(x)dx.
\end{equation}
\end{definition}

It is clear, and this is a key point, that the $A_\infty$ class of weights is contained in  $C_p$ for any $p\in(1,\infty)$.


We call the quantity $\int_{\R^n} (M\chi_{\strt{1.5ex}Q})^pw$ the $C_p$-tail of $w$ at $Q$. A weight has either finite $C_p$-tails at every cube or infinite $C_p$-tails at every cube. 


\begin{example}[\cite{Buckley-thesis}, Chapter 7]
Let $w \in A_p$ and $g$ a non-negative bounded convexly contoured function. Then $gw \in C_p$.
The weights in $C_p$ are non-doubling, and they may even vanish in a set of positive measure. 
\end{example}

The weights in this class also satisfy a non-local weak Reverse H\"older Inequality, as stated in the following proposition. We shall call this property Reverse H\"older Inequality (RHI) for $C_p$ weights, though it is not actually a proper RHI.

\begin{proposition}[Reverse H\"older Inequality for $C_p$ weights] \label{reverse-no-cuantificada}
A weight $w$ belongs to the class $C_p$ if and only if there exist $C,\delta>0$ such that for every cube $Q$ 
 \begin{equation}\label{RHIprimera}
\left(\avgint_Q w^{1+\delta}\right)^\frac1{1+\delta}\leq \frac C{|Q|} \int_{\R^n} (M\chi_{\strt{1.5ex}Q})^pw.
\end{equation}
Moreover, we have that $\delta$ in \eqref{RHIprimera} and $\eps$ in \eqref{Cp-definition} are equivalent up to a dimensional constant.
\end{proposition}
%


We present the sharp reverse H\"older inequality for $A_\infty$ weights. Using the notation in \cite{HPR1}, we define for a positive weight $w$
$$[w]_{A_\infty}:= \sup_Q \frac1{w(Q)}\int_Q M(w\chi_{\strt{1.5ex}Q}),$$where the supremum is taken over all cubes with sides parallel to the axes. It is known that $w\in A_\infty$ if and only if \,$[w]_{A_\infty}<\infty$.

\begin{theorem}[Sharp Reverse H\"older Inequality for $A_\infty$ weights, \cite{HPR1}] \label{SHARP-RHI-Theorem} Let $w\in A_\infty$ and let $Q$ be a cube. Then
\begin{equation}\label{SHARP-RHI}
\left( \avgint_{Q}w^{1+\delta} \right)^\frac{1}{1+\delta}\leq 2 \avgint_{Q}w,
\end{equation}
for any $\delta >0$ such that $0<\delta\leq \frac{1}{2^{n+1}[w]_{A_\infty}-1}.$
\end{theorem}

When we compare Proposition \ref{reverse-no-cuantificada} and Theorem \ref{SHARP-RHI-Theorem}, we notice that $\int_{\R^n} (M\chi_{\strt{1.5ex}Q})^pw$ in \eqref{RHIprimera} plays the role of $w(Q)$ in \eqref{SHARP-RHI}. Keeping this similarity in mind, we define the $C_p$ constant.

\begin{definition}[$C_p$ constant]
For an arbitrary non-zero weight $w$, we define
\begin{equation*}[w]_{C_p}:= \sup_Q \frac{1}{\int_{\R^n} (M\chi_{\strt{1.5ex}Q})^pw }\int_Q M(\chi_{\strt{1.5ex}Q}w),
\end{equation*}
where the supremum is taken over all cubes $Q$ with sides parallel to the axes. 
\end{definition}

Notice that if $w$ is not identically zero, the quantity on the denominator is always strictly greater than zero. 

\begin{remark}A weight $w$ has infinite $C_p$-tails if and only if $[w]_{C_p}=0$. Indeed, if $w$ has infinite $C_p$-tails then the denominator equals infinity and we have $[w]_{C_p}=0$. Conversely, if $[w]_{C_p}=0$ we have that for every cube $Q$, $$\frac{1}{\int_{\R^n} (M\chi_{\strt{1.5ex}Q})^pw }\int_Q M(\chi_{\strt{1.5ex}Q}w) =0.$$ This means that either $\int_Q(M\chi_{\strt{1.5ex}Q}w)=0$ or $\int_{\R^n} M (\chi_{\strt{1.5ex}Q})^pw =\infty$ for every cube $Q$. In the latter case, $w$ has infinite $C_p$-tails. If $\int_Q(M\chi_{\strt{1.5ex}Q}w)=0$ for every cube, then $w$ must be zero almost everywhere.\end{remark}

By Proposition \ref{reverse-no-cuantificada} we have that a weight $w$ is in the class $C_p$ if and only if $0\leq [w]_{C_p}<\infty$.

%

\begin{example}
For $p>1$ and small $\eps$, for $w_\eps(x) = |x|^{n(p-1-\eps)}$ we have $[w_\eps]_{C_p}\lesssim \eps.$ 
This can be shown by direct computation.
\end{example}

This is the main difference between the $A_\infty$ and $C_p$ constants, since $[w]_{A_\infty}\geq 1$ for an arbitrary  weight $w$.

\begin{remark}
For any weight $w$ we have the following relation between the different constants for $q\leq p$, $[w]_{C_q} \leq [w]_{C_p} \leq [w]_{A_\infty}.$
\end{remark}

We now restate the quantitative RHI for $C_p$ weights we mentioned on the introduction.

\begin{theorem}[Quantitative RHI for $C_p$ weights]
\label{RHI-Cp-Theorem}
Let $1<p<\infty$ and let $w$ be a weight such that $0\leq [w]_{C_p}<\infty.$
Then $w\in C_p$ and $w$ satisfies, for $\delta = \frac{1}{ B \max([w]_{C_p},1)}$, with 
$$B=\frac{2^{1+4np+3n}(20)^n}{1-2^{-n(p-1)}},$$
\begin{equation}\label{RHIconeldos}\left( \avgint_Qw^{1+\delta}\right)^\frac1{1+\delta} \leq  \frac{4}{|Q|}\int_{\R^n} (M\chi_{\strt{1.5ex}Q})^pw.
\end{equation}
\end{theorem}

\begin{remark}
Notice that $B$ depends on the dimension and on $p$. Moreover, we have $B\rightarrow \infty$ whenever $p$ tends to either $\infty$ or 1. 
\end{remark}

\begin{remark}
The quantification in terms of the parameters $\eps$  and $C$ in \eqref{Cp-definition} is $C=2$ and
$$\eps = \frac{1-2^{-n(p-1)}}{2^{2np+3n}(20)^n}\ \min(1,[w]_{C_p}^{-1}) .$$
In particular, we have that both $\eps$ and $\delta$ are smaller than one.
\end{remark}

\section{Proof of the RHI}

We may assume that $w$ has finite $C_p$-tails, that is, $[w]_{C_p}>0$. Indeed, if $[w]_{C_p}=0$ then the right side of \eqref{RHIconeldos} equals infinity and the theorem is trivially true.

The proof follows a remark from \cite{A-B-E-S}, section 8.1, keeping track of the dependence on the constant of the weight combined with the proof given in \cite{HPR1} of the RHI for $A_\infty$ weights. 

We now introduce a functional over cubes that serves as a discrete analogue for the $C_p$-tail. Define, for a cube $Q$
\begin{equation}\label{COLA-definition}
a_{C_p}(Q) := \sum_{k=0}^\infty 2^{-n(p-1)k}\avgint _{2^kQ}w.
\end{equation}We note that $\alpha = \sum_{k\geq 0}2^{-n(p-1)k}=(2^{n(p-1)})'<\infty$ only depends on $n$ and $p$. In the following lemma we prove that the discrete and continuous $C_p$-tails are equivalent.

\begin{lemma}\label{lema-acp}Let $\beta = \sum_{l=0}^\infty 2^{-npl}.$ Then, for every weight $w$ and every cube $Q$, we have
\begin{equation}\label{discretetail}
\frac 1\beta\: a_{C_p}(Q) \leq  \frac{1}{|Q|} \int_{\R^n} (M\chi_{\strt{1.5ex}Q})^pw \leq \frac{4^{np}}\beta \:a_{C_p}(Q).
\end{equation}
\end{lemma}

As a corollary of this, we have that $a_{C_p}(Q)<\infty$ for every cube $Q$ whenever $w$ has finite $C_p$-tails.

\begin{proof}

Observe that $\beta = \sum_{l=0}^\infty 2^{-npl}=(2^{np})'$  and hence $\beta<2.$ Note that for  $x\in2^kQ\setminus 2^{k-1}Q$ we have $2^{-kn} \leq M\chi_{\strt{1.5ex}Q}(x) \leq 2^{-n(k-2)}$. Then
\begin{align*}
\frac{1}{|Q|}\int_{\R^n} (M\chi_{\strt{1.5ex}Q})^pw & = \avgint_Q w +\sum_{k=1}^\infty \frac{1}{|Q|}\int_{2^kQ\setminus2^{k-1}Q} (M\chi_{\strt{1.5ex}Q})^pw ,
\end{align*} 
so we actually have
\begin{align*}
\avgint_Q w +\sum_{k=1}^\infty \frac{2^{-npk}}{|Q|}w({2^kQ\setminus2^{k-1}Q}) & \leq \frac{1}{|Q|}\int_{\R^n} (M\chi_{\strt{1.5ex}Q})^pw \\
& \leq  \avgint_Q w +\sum_{k=1}^\infty \frac{2^{-np(k-2)}}{|Q|}w(2^kQ\setminus2^{k-1}Q) \\
& \leq 4^{np} \left(  \avgint_Q w +\sum_{k=1}^\infty \frac{2^{-npk}}{|Q|}w(2^kQ\setminus2^{k-1}Q) \right)
\end{align*}
Now we rewrite \eqref{COLA-definition} in the following way
\begin{align*}
\sum_{k=0}^\infty 2^{-n(p-1)k}\avgint _{2^kQ}w &= \avgint_Qw+ \sum_{k=1}^\infty  \frac{2^{-npk}}{|Q|} \left( \int_Qw+\sum_{j=1}^{k}\int_{2^jQ\setminus 2^{j-1}Q}w \right) \\
& = \beta\avgint_Qw+\frac{1}{|Q|}\sum_{j=1}^\infty \left(\sum_{k=j}^\infty 2^{{-npk}} \right) \int_{2^jQ\setminus 2^{j-1}Q}w \\
& = \beta\left( \avgint_Qw + \frac{1}{|Q|} \sum_{j=1}^\infty 2^{-pnj} \int_{2^jQ\setminus 2^{j-1}Q}w\right).
\end{align*}
This finishes the proof of \eqref{discretetail}.
\end{proof}

\begin{proposition}\label{porpositioncopiada}
Let $w$ be a weight and $p>1$. Suppose that there exists a constant $0<\gamma<\infty$ such that for every cube $Q$
\begin{equation}\label{HIPOTESIS-CP}
\avgint_Q M(\chi_{\strt{1.5ex}Q} w) \leq \gamma \, a_{C_p}(Q)<\infty.
\end{equation}
Then there exists $0<\delta\leq \frac1{A \max(\gamma,1)}$, with 
$$A=20^n\frac{2^{1+3n}}{1-2^{-n(p-1)}},$$ such that for every cube $Q$,
\begin{equation*}\label{RHI-paralaM}
 \avgint_Q M(\chi_{\strt{1.5ex}Q}w)^{1+\delta} \leq 2^{1+n(2p+3)}\: \gamma \: a_{C_p}(Q)^{1+\delta}.
\end{equation*}
\end{proposition}

Note that the infimum of the constants $\gamma$ such that \eqref{HIPOTESIS-CP} holds is equivalent to the $C_p$ constant of $w$, because of Lemma \ref{lema-acp}. In this case we will have $0<[w]_{C_p}<\infty$.

\begin{proof}
Fix a cube $Q=Q(x_0,R)$, that is, the cube centered at the point $x_0$ and with side length $2R$ ($Q(x,R)$ is just a ball with the $l^\infty$ distance in $\R^n$). The proof will be carried out following some steps.

{\it Step 1.} Let $r, \rho>0$ and $l\in \Z$ be numbers that satisfy $R\leq r<\rho\leq 2R$ and $2^l(\rho-r)=R$. This in particular implies $l\geq 0$. 

We define a new maximal operator 
$$\tilde Mv(x) := \sup_{k\in \Z} \avgint_{Q(x,2^k(\rho-r))}|v|.$$
We have the following pointwise bounds between the different maximal functions
$$\tilde M v \leq Mv \leq \kappa \tilde Mv,$$ where $\kappa$ does not depend on $\rho-r$. In particular, we can choose $\kappa = 4^n$.
For $t\geq 0$ and a function $F$ we define $F_t = \min(F,t)$. Now fix $m>0$ with the intention of letting $m\rightarrow \infty$ in the end. Call $Q_{r}=Q(x_0,r)$ and $Q_{\rho}=Q(x_0,\rho).$



 We then have 
 \begin{align*} \int_{Q_{r}} (M(\chi_{\strt{1.5ex}Q_{r}}w))^{1+\delta}_m & \leq \kappa ^{1+\delta} \int_{Q_{r}}(\tilde M(\chi_{\strt{1.5ex}Q_{r}   }w))^{\delta}_{m}\:\tilde M (\chi_{\strt{1.5ex}Q_{r}}w) \\
& \leq \kappa^{1+\delta}\int_{Q_{r}}(\tilde M(\chi_{\strt{1.5ex}{Q_{\rho}}}w))^{\delta}_{m}\: \tilde M (\chi_{\strt{1.5ex}Q_{\rho}}w) \\
& \leq \kappa^{1+\delta}\delta \int_0^{m} \la^{\delta-1}u(Q_{r}\cap \{u>\la\})d\la,
\end{align*}
where $u= \tilde M (\chi_{\strt{1.5ex}Q_{\rho}}w).$ To state it in a separate line, we have
\begin{equation}\label{STEP3}
\int_{Q_{r}} (M(\chi_{\strt{1.5ex}Q_{r}}w))^{1+\delta}_m \leq  \kappa^{1+\delta}\delta \int_0^{m} \la^{\delta-1}u(Q_{r}\cap \{u>\la\})d\la.
\end{equation}

{\it Step 2.}
Now we pick $\la_0:= 2^{n(l+1)} a_{C_p}(2Q)$ (which is finite by hypothesis). It is easy to see that for $x\in Q_{r}$ and $k\geq0$, by the choice of $\la_0$ we have \begin{equation} \label{STEP4}
\avgint_{Q(x,2^k(\rho-r))}\chi_{\strt{1.5ex}Q_{\rho}}w \leq  \la_0.
\end{equation}

Indeed, we have that $Q_\rho \subset 2Q$, so we can make
\begin{align*}
\avgint_{Q(x,2^k(\rho-r))}\chi_{\strt{1.5ex}Q_{\rho}}w & \leq \avgint_{Q(x,2^k(\rho-r))}\chi_{\strt{1.5ex}2Q}w \\
& = \frac{|2Q|}{|Q(x,2^k(\rho-r))|}\avgint_{2Q}w \\
& \leq 2^{n(l+1-k)} a_{C_p}(2Q) \leq 2^{n(l+1)}a_{C_p}(2Q).
\end{align*}

%
%
This completes the proof of \eqref{STEP4} when $x\in Q_r$ and $k\geq0$.

Let $\la>\la_0$ and $x \in Q_{r}\cap \{u>\la\}$. As $u(x) =\tilde M(\chi_{\strt{1.5ex}Q_{\rho}}w)(x)>\la>\la_0$, \eqref{STEP4} and the fact $Q(x,2^k(\rho-r))\subset Q_{\rho}$ when $k<0$ imply
$$u(x) = \sup_{k<0}\avgint_{Q(x,2^k(\rho-r))}\chi_{\strt{1.5ex}Q_{\rho}}w \:= \sup_{k<0} \avgint_{Q(x,2^k(\rho-r))}w.$$ 
For such an $x$, let $k_x = \max\{k : \: \avgint_{Q(x,2^k(\rho-r))}w >\la\}$. Trivially,  we have $$Q_r \cap\{u>\la\} \subset \bigcup_{x\in Q_r \cap\{u>\la\}} Q(x,\frac15 2^{k_x}(\rho-r)).$$
We use the Vitali covering lemma for infinite sets and choose a countable collection of $x_i\in Q_r \cap\{u>\la\}$ so that the family of cubes $Q_i = Q(x_i,2^{k_{x_i}}(\rho-r))$ satisfy the following properties:
\begin{itemize}
\item $Q_r \cap\{u>\la\} \subset \cup_i Q_i,$
\item the cubes $\frac15 Q_i$ are pairwise disjoint,
\item $\avgint_{Q_i}w >\la$,
\item $\avgint_{2^k Q_i}w \leq \la$, for any $k\geq 1$ 
\item $Q_i\subset Q_\rho$.
\end{itemize} 

We make the following claim. If we denote $Q_i^*= 2Q_i$ then for all $x\in Q_i\cap Q_{r}$, $$u(x)\leq 2^{n}M(\chi_{\strt{1.5ex}Q_i^*}w)(x).$$

Indeed, fix $x \in Q_i\cap Q_{r}$ and $k<0$. If $k\geq k_{x_i}$ then by the stopping time we get
\begin{align*}
\avgint_{Q(x,2^k(\rho-r))}w & \leq \frac{|Q(x_i,2^{k+1}(\rho-r))|}{|Q(x,2^{k}(\rho-r))|}\avgint_{Q(x_i,2^{k+1}(\rho-r))}w \\
&\leq 2^n \la \leq 2^n \avgint_{Q_i}w \leq 2^n M(\chi_{\strt{1.5ex}Q_i^*}w)(x).
\end{align*}
In the other case, namely $k< k_{x_i}$  we have $Q(x,2^k(\rho-r))\subset Q_i^*\cap Q_{\rho}$ and hence
$$\avgint_{Q(x,2^k(\rho-r))} w  \leq M(\chi_{\strt{1.5ex}Q_i^*}w)(x),$$ and thus the claim is proved.

{\it Step 3.}
We use now this claim together with the stopping time and the hypothesis \eqref{HIPOTESIS-CP} to see
\begin{align*}
u(Q_{r}\cap \{u>\la\}) & \leq \sum_i u(Q_i\cap Q_{r}) \leq \sum_i \int_{Q_i\cap Q_{r}}u  \leq 2^n \sum_i \int_{Q_i\cap Q_{r}}M(\chi_{\strt{1.5ex}Q_i^*}w)\\
& \leq 2^n \sum_i |Q_i^*| \avgint_{Q_i^*}M(\chi_{\strt{1.5ex}Q_i^*}w)  \leq 2^n \gamma  \sum_i |Q_i^*| a_{C_p}(Q_i^*) \end{align*}
But, using the properties of $Q_i$ we get
$$a_{C_p}(Q_i^*) =\sum_{k=0}^\infty 2^{-nk(p-1)}\avgint_{2^{k+1}Q_i}w \leq \la \alpha,$$ so we have
\begin{align*}
u(Q_{r}\cap \{u>\la\}) &  \leq 2^n \gamma \sum_i |Q_i^*| \alpha \la 
 \leq (20)^n \gamma \alpha \: |\cup_iQ_i| \la,\end{align*}
where in the last inequality we have used that $\frac15 Q_i$ are disjoint. Since each one of the cubes $Q_i\subset Q_{\rho}$ and $\la <\avgint_{Q_i}w$ we have $\cup_iQ_i \subset Q_{\rho}\cap\{M(\chi_{\strt{1.5ex}Q_{\rho}}w)>\la\}$ so we have obtained for $\la >\la_0$
\begin{equation*}\label{STEP5}
u(Q_{r}\cap\{u>\la \}) \leq (20)^n \alpha \gamma \la |Q_{\rho}\cap\{M(\chi_{\strt{1.5ex}Q_{\rho}}w)>\la\}|.
\end{equation*}
Plugging everything on what we had in \eqref{STEP3} we have
\begin{align*}
\int_{Q_{r}} (M(\chi_{\strt{1.5ex}Q_{r}}))^{1+\delta}_m & \leq  \kappa^{1+\delta}\la_0^\delta u(Q_{r})  + \kappa^{\delta+1} (20)^n \gamma \alpha  \delta \int_{\la_0}^{m} \la^{\delta}|Q_{\rho}\cap \{M(\chi_{\strt{1.5ex}Q_{\rho}}w)>\la\}|d\la.
\end{align*}

{\it Step 4.} We define
$$\FI(t) = \int_{Q_t}  (M(\chi_{\strt{1.5ex}Q_t}w))^{1+\delta}_m  \qquad t>0.
$$ 
Observe that $\FI(t)  < \infty$ for any $t>0$. We claim that,
\begin{equation}\label{iteration}
\FI(r) \leq 
c_1\gamma |Q|2^{nl\delta}\left(a_{C_p}(Q)\right)^{1+\delta} + \delta \,\kappa^{\delta+1} (20)^n \gamma \alpha \FI(\rho).
\end{equation}
Indeed, combining what we obtained before in the following way:
\begin{align*}
\FI(r) & \leq c_1\gamma|Q|2^{nl\delta}\left(a_{C_p}(Q)\right)^{1+\delta} + \kappa^{\delta+1} (20)^n \gamma \alpha  \frac{\delta }{\delta+1} \int_{Q_{\rho}} M(\chi_{\strt{1.5ex}Q_{\rho}}w)^{\delta+1}_{m} \\
& \leq c_1\gamma |Q|2^{nl\delta}\left(a_{C_p}(Q)\right)^{1+\delta} + ( \kappa^{\delta+1} (20)^n \gamma \alpha) \delta \FI(\rho),
\end{align*}
where $c_1=2^{n(p+1)(\delta+1)}$, and where we have used
$$u(Q_r) = \int_{Q_r}\tilde M(\chi_{\strt{1.5ex}Q_\rho}w) \leq |2Q| \avgint _{2Q}M(\chi_{\strt{1.5ex}2Q}w) \leq 2^n |Q| \gamma a_{C_p}(2Q)\leq 2^{np} |Q|\gamma a_{C_p}(Q),$$
since
$$a_{C_p}(2Q)\leq 2^{n(p-1)}a_{C_p}(Q).$$
This yields the claim.

\textit{Step 5.} Now we present an iteration scheme  starting from claim \eqref{iteration}. Remember that $l\geq 0$ was an integer such that $2^l(\rho-r)=R$. Set 
\begin{align*}
t_0 &= R,\\
t_{i+1}&=t_i+2^{-(i+1)}R= \sum_{j=0}^{i+1}2^{-j}R, \quad i\geq 0.
\end{align*}
Clearly, $t_i\rightarrow 2R$ as $i\rightarrow \infty$. This way, $2^{i+1}(t_{i+1}-t_i)=R$ and we can use them as $\rho=t_{i+1}$, $t_i = r$, and $l=i+1$   in \eqref{iteration}.

In other words, we have the estimate for $\FI(t_i)$ in terms of $\FI(t_{i+1})$:
$$\FI(t_i) \leq c_2 2^{n\delta i}+ c_3  \FI(t_{i+1}),$$
where $c_2 =c_1 2^{n\delta} \gamma |Q| (a_{C_p}(Q))^{1+\delta},$ $c_3 =\kappa^{\delta+1} 20^n \alpha  \gamma  \delta$. So, iterating this last inequality $i_0$ times we get
\begin{align*}
\FI(R)&=\FI(t_0)\leq c_2 \sum_{j=0}^{i_0-1} (c_3 2^{n\delta})^j + c_3 ^{i_0}\FI(t_{i_0}) 
\leq c_2 \sum_{j=0}^{i_0-1} (c_3 2^{n\delta})^j + (c_3)^{i_0}\FI(2R)
\end{align*}

We have to choose $\delta>0$ so that we have the relation
\begin{equation} \label{elecciondedelta}
c_3 2^{n\delta}= 20^n \kappa^{\delta+1} \gamma \alpha\delta2^{n\delta}<1/2.
\end{equation}
We may suppose $\delta<1$. Once we have \eqref{elecciondedelta}, we can take the limit $i_0\rightarrow \infty$ and the sum is bounded by 2 and the second term goes to zero since $\FI(2R)<\infty$. Hence
\begin{align*}
\FI(R) & \leq  2c_2  = 2^{1+n\delta+n(\delta+1)(p+1)} \gamma |Q| (a_{C_p}(Q))^{1+\delta}\\
 & <  2^{1+n(2p+3)} \gamma |Q| (a_{C_p}(Q))^{1+\delta},
\end{align*}
and then 
$$\frac{1}{|Q|} \int_{Q}M(\chi_{\strt{1.5ex}Q} w)^{1+\delta}_m \leq 2^{1+n(2p+3)}  \gamma \left( a_{C_p}(Q)\right)^{1+\delta}.$$
Now, letting $m\rightarrow \infty$ and using the Fatou lemma we can conclude the proof.

To finish the proof, we make the choice of $\delta$ as follows. Coming back to \eqref{elecciondedelta} we see that, since we have $\delta$ in the exponent and $\gamma$ can be arbitrarily small, we have to choose $\delta = \frac{1}{A\max(1,\gamma)}$ with \begin{equation*}A=2\kappa^2 (20)^n2^{n}\alpha= (20)^n\frac{2^{1+3n}}{1-2^{-n(p-1)}}.\qedhere\end{equation*}
\end{proof}

We are ready to finally prove the theorem.

\begin{proof}[Proof of Theorem \ref{RHI-Cp-Theorem}]
Fix a cube $Q$. Let $M_{d,Q}$ denote the maximal operator with respect to the dyadic children of $Q$, that is $$M_{d,Q}v(x) = \sup_{\substack{R\in \mathcal D(Q) \\ x \in R}}\frac{1}{|R|}\int_R|v|, \quad x \in Q.$$
We argue as in \cite{HPR1}, Theorem 2.3. By the Lebesgue differentiation theorem,
\begin{align*}
\int_Qw^{1+\delta} \leq \int_{Q}(M_{d,Q}w)^\delta w.
\end{align*}
Call now $\Omega_\la = \{x\in Q: M_{d,Q}w(x)>\la\}$. For $\la\geq w_Q$ we make the Calder\'on--Zygmund decomposition of $w$ at height $\la$ to obtain $\Omega_\la= \cup_j Q_j$ with $Q_j$ pairwise disjoint and $$\la <\frac{1}{|Q_j|} \int_{Q_j}w\leq 2^n \la.$$Multiplying by $|Q_j|$ and summing on $j$ this inequality chain becomes 
$$\la |\Omega_\la| \leq w(\Omega_\la) \leq 2^n \la |\Omega_\la|.$$

Then we have
\begin{align*}
\avgint_Q (M_{d,Q}w)^\delta w & = \frac{1}{|Q|}\int_0^\infty \delta \la^{\delta-1}w(\Omega_\la)d\la \\
& \leq w_{Q}^{\delta+1}+ \frac{1}{|Q|}\int_{w_Q}^\infty \delta \la^{\delta-1}w(\Omega_\la) d\la \\
& \leq w_{Q}^{\delta+1}+ \delta 2^n \frac{1}{|Q|}\int_{w_Q}^\infty  \la^{\delta}|\Omega_\la| d\la \\
& \leq w_{Q}^{\delta+1} + 2^n \frac{\delta}{\delta+1} \frac{1}{|Q|}\int_Q (M_{d,Q}w)^{1+\delta}.
\end{align*}
Now we apply Proposition \ref{porpositioncopiada}. We have $[w]_{C_p}\leq \beta \gamma \leq 4^{np}[w]_{C_p}$, so we need $\delta \leq \beta/A(\max(1,[w]_{C_p}),$ with $\beta$ as in Lemma \ref{lema-acp}. So we get 
\begin{align*}
\avgint_Q (M_{d,Q}w)^\delta w & \leq (1+2^{1+n(2p+4)} \frac{\delta}{\delta+1} \gamma )\left(a_{C_p}(Q)\right)^{1+\delta}\\
& \leq (1+2^{1+n(2p+4)} \frac{\delta}{\delta+1} [w]_{C_p} \frac{4^{np}}{\beta} )\left(\frac{\beta}{|Q|}\int (M\chi_{\strt{1.5ex}Q})^pw \right)^{1+\delta},
\end{align*}
where we have used Lemma \ref{lema-acp}. Now, since we have $2^{4np}/\beta$ multiplying $\delta$, we have to change the choice of $\delta$ slightly and make $$\delta \leq\frac{2^{-4np}}{\beta} \frac\beta{A\max(1,[w]_{C_p})} =   \frac1{B \max(1,[w]_{C_p})}.$$
This finishes the proof of the theorem.
\end{proof}

\section{Sharpness of the exponent} 

For a cube $Q$, it is clear that $M\chi_{\strt{1.5ex}Q}$ equals 1 on the cube and is smaller than 1 outside the cube. Therefore $(M\chi_{\strt{1.5ex}Q})^p$ converges to $\chi_{\strt{1.5ex}Q}$ a.e. when $p\rightarrow \infty$. Moreover, for a weight $w$ with finite $C_{p_0}$-tails, by the dominated convergence theorem we have
$$\lim_{p\rightarrow \infty} \int_{\R^n} (M\chi_{\strt{1.5ex}Q})^pw = w(Q).$$

For any weight $w\in A_\infty$, we have by the definition of the constant $[w]_{A_\infty}$ that for any cube $Q$ 
$$\int_Q M(w\chi_{\strt{1.5ex}Q}) \leq [w]_{A_\infty}w(Q) \leq[w]_{A_\infty} a_{C_p}(Q),$$
where  $a_{C_p}(Q)=\sum_{k\geq 0}2^{-n(p-1)k}\avgint_{2^kQ}w$ is the discrete $C_p$-tail introduced in the previous section.

If we modify slightly the proof of Proposition \ref{porpositioncopiada} and Theorem \ref{RHI-Cp-Theorem} and add some extra hypothesis, we can recover the RHI for $A_\infty$ weights. We explain how to do this in this section.

Fix a number $s>1$. This will be the dilation parameter, which was $s=2$ in the previous section. We plan on letting $t$ tend to one in the end. We introduce the corresponding discrete $C_p$-tail with respect to $t$,
$$a_{C_p,s} (Q)= \sum_{k\geq0}s^{-n(p-1)k} \avgint_{s^k Q}w.$$
Note that for any weight $w\in C_{p_0}$ we have $\lim_{p\rightarrow \infty} a_{C_p,s}(Q) = w_Q$ for any $s>1$. Also, for a fixed $s>1$ we introduce the corresponding discrete $C_p$ constant $$[w]_{C_p,s} := \sup_Q \frac{\int_Q M(\chi_{\strt{1.5ex}Q} w)}{a_{C_p,s}(Q)}$$

\begin{remark}\label{REMARKCP}
For a weight $w \in A_\infty$ and any $s>1$ we have $\lim_{p\rightarrow\infty}[w]_{C_p,s} \leq [w]_{A_\infty}$.
\end{remark}


\begin{theorem}\label{theorem-dilataciones} Fix $2\geq s>1$ and $1<p<\infty$.
For a weight $w$ in $C_p$ and $\delta = \frac{1}{A_{t,p}\max(1,[w]_{C_p,s})}$ and every cube $Q$, with $$A_{s,p}=\frac{5^n 2^{1+5n}}{1-s^{-n(p-1)}},$$ we have
\begin{equation}\label{equationcondilataciones}\left( \frac{1}{|Q|} \int_Q w^{1+\delta}\right)^\frac{1}{1+\delta} \leq ({2^n+1})\ a_{C_p,s}(s Q).\end{equation}
\end{theorem}

Before we prove this theorem, we give a proof of Theorem \ref{SHARP-RHI-Theorem} as a corollary. Let $w\in A_\infty$. By Remark \ref{REMARKCP}, we can let $p\rightarrow \infty$ in equation \eqref{equationcondilataciones} and we obtain
\begin{equation}\label{atomarelimite}\left( \frac{1}{|Q|} \int_Q w^{1+\delta_\infty}\right)^\frac{1}{1+\delta_\infty} \leq ({2^n+1})\ w_{s Q},
\end{equation}
where 
$$\delta_\infty = \lim_{p\rightarrow \infty} \frac{1-s^{-n(p-1)}}{c_n \max(1,[w]_{C_p,s})} = \frac1{c_n [w]_{A_\infty}}.$$
Now we let $s \rightarrow 1$ in \eqref{atomarelimite} and obtain $$\left( \frac{1}{|Q|} \int_Q w^{1+\delta_\infty}\right)^\frac{1}{1+\delta_\infty} \leq ({2^n+1})\ w_{ Q},$$ which is in fact the reverse H\"older inequality for $A_\infty$ weights.

\begin{remark} The dimensional constants are bigger from those in Theorem \ref{SHARP-RHI-Theorem}, but the dependence on the weight is essentially the same. Because of this, we obtain that the dependence on $w$ in Theorem \ref{RHI-Cp-Theorem} is sharp.\end{remark}

\begin{proof}[Proof of Theorem \ref{theorem-dilataciones}]
We repeat the first three steps of the proof of Proposition \ref{porpositioncopiada}, with the following modifications. This time, $r,\rho,l$ will satisfy $s^l(\rho-r)=R$ and $R\leq r<\rho\leq R$. Also, now we will use the maximal operator $\tilde Mv(x) = \sup_{k\in \Z}\avgint_{Q(x,s^k(\rho-r))}u,$ and some other trivial changes. For the fourth step, we leave $a_{C_p,s}(s Q)$ in the equation, so we get
$$\FI(r)  \leq s^{n(\delta+1)}\gamma |Q|s^{n\delta l}\left(a_{C_p,s}(s Q)\right)^{1+\delta} + ( \kappa^{1+\delta} (5s^2)^n \gamma \alpha_s) \ \delta \ \FI(\rho),
$$
where $\alpha_s = \sum_{k\geq0}s^{-nk(p-1)}=(1-s^{-n(p-1)})^{-1}.$
We make a similiar iteration scheme, namely $t_0=R$ and $t_{i+1}=t_i+s^{-(i+1)}R \leq s R.$ Now the condition for $\delta$ translates to $\delta \leq \frac 1{A_{s,p}\max(1,\gamma)}$ where $$A_{s,p}=\frac{5^{n} 2^{1+5n}}{1-s^-{n(p-1)}}.$$ 
The main difference is that now we get
\begin{equation*}\frac{1}{|Q|}\int_Q (M(\chi_{\strt{1.5ex}Q}w)_m)^{1+\delta} \leq 2^{1+5n} \gamma \big(a_{C_p,s}(s Q)\big)^{1+\delta},\end{equation*}
where the right part stays bounded whenever $p\rightarrow \infty$. Now we use Fatou lemma and make $m\rightarrow \infty$ to get
\begin{equation}\label{equationextrana}\frac{1}{|Q|}\int_Q M(\chi_{\strt{1.5ex}Q}w)^{1+\delta} \leq 2^{1+5n} \gamma \big(a_{C_p,s}(s Q)\big)^{1+\delta}.\end{equation}

Now we make the argument in the proof of Theorem \ref{RHI-Cp-Theorem} and combine it with \eqref{equationextrana}. We get, 
\begin{align*}
\avgint_Q w^{1+\delta} & \leq (w_Q)^{1+\delta}+2^n \frac \delta{1+\delta} \frac 1{|Q|}\int_Q(M_{d,Q}w)^{1+\delta}\\
& \leq (w_Q)^{1+\delta}+2^n \frac \delta{1+\delta} 2^{1+5n} \gamma \big(a_{C_p,s}(s Q)\big)^{1+\delta}\\
& \leq (2^n +\delta 2^{1+6n} \gamma) \big(a_{C_p,s}(s Q)\big)^{1+\delta}\\
& \leq (2^n +1) \big(a_{C_p,s}(s Q)\big)^{1+\delta},
\end{align*}
whenever $\delta \leq \frac1{2^{1+6n}\gamma}$, which is true by the choice of $\delta$. This finishes the proof.
\end{proof}

\section{$C_p$ weights and the Coifman--Fefferman inequality}

Let $T^*$ denote a maximally truncated Calder\'on--Zygmund operator and $M$ the Hardy--Littlewood maximal operator. Then for $w\in A_\infty$ and any $f\in L^\infty_c$, we have for any $0<p<\infty$
\begin{equation}\label{Ecuacionconpesoclasica}
\| T^* f\|_{L^p(w)} \leq C \: \|Mf\|_{L^p(w)},
\end{equation}
where the constant depends only on $w$, $T$ and $p$.

The classical proof of inequality \eqref{Ecuacionconpesoclasica} in \cite{CoifmanFeffermanpaper} uses a good-$\la$ inequality between the operators $T^*$ and $M$. If the kernel of $T$ is not regular enough, there is in general no good-$\la$ inequality and even inequality \eqref{Ecuacionconpesoclasica} can be false, as is shown in \cite{MPTG}. 

There are ways of proving inequality \eqref{Ecuacionconpesoclasica} without using the good-$\la$ inequality. For example, the proof given in \cite{AP} uses a pointwise estimate involving the sharp maximal function. Another proof can be found in \cite{CMP}, where the main tool is an extrapolation result that allows to obtain estimates like \eqref{Ecuacionconpesoclasica} for any $A_\infty$ weight from the smaller class $A_1$ (see also \cite{CGMP}).

Inequality \eqref{Ecuacionconpesoclasica} is very important in the classical theory of Calder\'on--Zygmund operators, as it is used in the proof of many other weighted norm inequalities. The first, and probably most important consequence of \eqref{Ecuacionconpesoclasica} is the boundedness of $T^*$ in $L^p(w)$ for any weight $w\in A_p$, $1<p<\infty$, namely
$$\int_{\R^n} (T^*f)^p w \leq c \int_{\R^n} |f|^p w.$$ This comes as a direct corollary of Muckenhoupt's theorem \cite{Muck}. 

Another consequence of inequality \eqref{Ecuacionconpesoclasica}, though not as direct as the previous one, is the following inequality, obtained in \cite{P0}. For any weight $w$ it holds
$$\norm{T^*f}_{L^p(w)} \leq c \norm f_{L^p(M^{[p]+1}w)},$$
where $[p]$ denotes the integer part of $p$ and $M^k$ denotes the $k-$fold composition of $M$. This result is sharp since $[p]+1$ cannot be replaced by $[p]+1$. This is saying that inequality \eqref{Ecuacionconpesoclasica} encodes a lot of information.  Very recently, this result was extended in \cite{LPRR} to the non-smooth case kernels, more precisely to the case case of rough singular operators  $T_{\Omega}$ with $\Omega \in L^{\infty}(\mathbb{S}^{n-1})$, by proving inequality \eqref{Ecuacionconpesoclasica} for these operators.   The proof of this result is quite different from the classical situation since there is no good-$\lambda$ estimate involving these operators and it is a consequence of a sparse domination result for $T_{\Omega}$ obtained in \cite{CACDPO} combined with the $A_{\infty}$ extrapolation theorem mentioned above in \cite{CMP}.

Norm inequalities similar to \eqref{Ecuacionconpesoclasica} are true for other operators, for instance in \cite{MW} (fractional integrals) or \cite{Wilson} (square functions).
Also, in the context of multilinear  harmonic analysis one can find other examples, for example, it was shown in \cite{LOPTT} an analogue for multilinear Calder\'on--Zygmund operators $T$, namely
$$\norm{T(f_1,...,f_m)}_{L^p(w)} \leq c \norm{\mathcal M(f_1,...,f_m)}_{L^p(w)},$$
for $w\in A_\infty$ extending \eqref{Ecuacionconpesoclasica}. We refer to \cite{LOPTT} for the definition of the operator $\mathcal M$. The proof for the multilinear setting is in the spirit of the proof of inequality \eqref{Ecuacionconpesoclasica} given in  \cite{AP}. There are also inequalities for \eqref{Ecuacionconpesoclasica} for more singular operators like the case of  commutators of Calder\'on--Zygmund operators with BMO functions, as was proved in \cite{P1}. In this case, the result is, for $w\in A_\infty,$ $$\norm{[b,T]f}_{L^p(w)}\leq c \norm{M^2f}_{L^p(w)},$$ where $[b,T]f = bTf-T(bf)$ and $M^2=M\circ M$. The result is false for $M$, because the commutator is not of weak type (1,1) and it would then contradict the extrapolation result from \cite{CMP}. 
 
All of the inequalities mentioned above are true for the class $A_\infty$ of weights, but $A_\infty$ is not the whole picture for some of them. The correct class of weights is, in some sense, the $C_p$ class. Muckenhoupt showed in \cite{Muckenhoupt1981} that $A_\infty$ is not necessary for the CFI \eqref{Ecuacionconpesoclasica}, and that the correct necesary condition is $C_p$. About sufficiency, Sawyer \cite{Sawyer1} proved that $w\in C_{p+\eta}$ for some $\eta>0$ is sufficient for \eqref{Ecuacionconpesoclasica} in the range $p\in(1,\infty)$. It is still an open conjecture if $C_p$ is a sufficient condition.

Although $C_p$ weights were introduced in the context of the CFI, other inequalities have been proved to hold for these weights. For example, the Fefferman--Stein inequality, between the maximal operators of Hardy--Littlewood and of Fefferman--Stein, as can be found in \cite{Yabuta}, \cite{CP} for a quantified version, \cite{LernerCP} in the weak-type context. In \cite{CLPR}, the authors extended Sawyer's result to a wider class of operators than Calder\'on--Zygmund operators, including some pseudo-differential operators and oscillatory integrals. Finally, in \cite{CLRT}, Sawyer's result was extended to rough singular integrals and sparse forms.

The rest of this Section is devoted to the quantification of Sawyer's result.
We define now the Calder\'on--Zygmund operators in a similar way as in \cite{CoifmanFeffermanpaper}. We will need a kernel $K$ defined away from the diagonal $x=y$ of $(\R^n)^{2}$ that satisfies the size condition
$$|K(x,y)|\leq \frac{A}{|x-y|^{n}}$$
for some $A>0$ and every $x\neq y$. Furthermore, we require the following regularity conditions for some $\eps>0$
\begin{align*}
|K(x,y)-K(x',y)|&\leq A\frac{|x-x'|^\eps}{|x-y|^{n+\eps}}
\end{align*}
whenever $2|x-x'|\leq |x-y|$, and the symmetric condition
\begin{align*}
|K(x,y)-K(x,y')|&\leq A\frac{|y-y'|^\eps}{|x-y|^{n+\eps}}
\end{align*}
whenever  $2|y-y'|\leq |x-y|$.

A Calder\'on--Zygmund operator associated to a kernel $K$ satisfying the above conditions is a linear operator $T:\mathscr S (\R^n) \longrightarrow \mathscr S'(\R^n)$ that satisfies 
$$Tf(x) = \int_{\R^n}K(x,y)f(y)dy,$$
for $f \in C_c^\infty(\R^n)$ and $x\not \in \text{supp}(f)$. Additionally, we will require that $T$ is bounded in $L^2$.

Now we define the maximal truncated singular integral operator $T^*$ as follows 
$$T^*f(x)=\sup_{\eps>0} \left| \int_{|x-y|>\eps} K(x,y)f(y)dy\right|.$$

We state the  quantification of Theorem B from \cite{Sawyer1} and Theorem 16 from \cite{CLPR}.

\begin{theorem}\label{norminequalitythm}
Fix $q>p>1$. For all Calder\'on--Zygmund operator $T$, all bounded $f$ with compact support and all weights $w\in C_q$ we have
$$\norm{T^*f}_{L^p(w)}\leq  c_{n,T}\ (q+\frac{qp^2}{q-p}) \ \Phi( \max([w]_{C_p},1)) \norm{ M  f}_{L^p(w)},$$
where $\Phi(t) =t\log(e+ t).$
\end{theorem}

We begin with a few lemmas, which correspond to Lemmas 2-4 in \cite{Sawyer1}. We include most of the details concerning the quantification of the weight for the sake of completion.

\begin{lemma}\label{Lemma3Sawyer}
Let $w\in C_q$. Fix $R\geq 2$ and $\delta>0$. Then for every cube $Q$ and any collection of pairwise disjoint cubes $Q_j\subset Q$ we have 
\begin{equation}\label{eq14-Sawyer}
\int_{RQ}\sum_j (M\chi_{\strt{1.5ex}Q_j}(x))^pw(x)dx \leq \frac 1{a\eps}\log \frac{c R^{nq}}{\eps \delta} w(RQ) + \delta \int_{\R^n} M\chi_{\strt{1.5ex}Q}(x)^q w(x)dx,
\end{equation}
where $a,c$ are dimensional constants and $\eps$ is the parameter for $w$ in \eqref{Cp-definition}. Hence, we have
\begin{equation}\label{eq15-Sawyer}
\int \sum_j (M\chi_{\strt{1.5ex}Q_j}(x))^qw(x)dx \leq c_n  4^{nq}\frac 1\eps \int_{\R^n} (M\chi_{\strt{1.5ex}Q}(x))^qw(x)dx.
\end{equation}
\end{lemma}

\begin{proof}
For $\la>0$, we will call $E_\la = \{x \in RQ: \sum_j M\chi_{\strt{1.5ex}Q_j}(x)^q>\la\}.$ Since the cubes are pairwise disjoint, we have $\sum_j \chi_{\strt{1.5ex}Q_j}\in L^\infty.$ Then by the exponential inequality from \cite{FeffermanStein} we have $|E_\la|\leq c_n e^{-a\la}|RQ|,$ where $c_n$ and $a$ are positive dimensional constants. Then, applying the $C_q$ condition \eqref{Cp-definition} we get 
\begin{align*}
w(E_\la) &\leq 2 \left(\frac{|E_\la|}{|RQ|}\right)^\eps \int_{\R^n} (M\chi_{\strt{1.5ex}RQ}(x))^qw(x)dx\\
& \leq c_n e^{-\eps a \la} R^{nq}\int_{\R^n} (M\chi_{\strt{1.5ex}Q}(x))^q w(x)dx.
\end{align*}
Now we compute
\begin{align*}
\int_{RQ} \sum_j (M\chi_{\strt{1.5ex}Q_j}(x))^qw(x)dx & = \int_0^\infty w(E_t) dt= \la w(E_\la) + \int_\la^\infty w(E_t)dt \\
& \leq \la w(RQ) + c_n R^{qn} \frac{1}{a\eps} e^{-a\eps\la} \int_{\R^n} (M\chi_{\strt{1.5ex}Q}(x))^qw(x)dx.
\end{align*}
We can choose $\la$ big enough so that 
$$c_n R^{qn} \frac{1}{a\eps} e^{-a\eps\la} \leq \delta,$$
and we get \eqref{eq14-Sawyer}. In order to get \eqref{eq15-Sawyer}, choose $R=2$, $\delta = \frac 1\eps$ and use $\sum M\chi_{\strt{1.5ex}Q_j}^q \leq 2^{nq} M\chi_{\strt{1.5ex}Q}$ almost everywhere outside of $2Q.$
\end{proof}


\begin{lemma}[Whitney covering lemma] Given $R\geq 1$, there is $C=C(n,R)$ such that if $\Omega$ is an open subset in $\R^n$, then $\Omega = \cup_j Q_j$ where the $Q_j$ are disjoint cubes satisfying
$$5R \leq \frac{\dist (Q_j,\R^n\setminus \Omega)}{\diam Q_j}\leq 15R,$$
$$\sum_j \chi_{\strt{1.5ex}RQ_j}\leq C\chi_{\strt{1.5ex}Q}.$$
\end{lemma}

We now define an auxiliary function considered in \cite{Sawyer1}. This operator will be used to intuitively represent the integral of the function $h$ to the power $p$ after we apply the $C_q$ condition.

\begin{definition}\label{Marcinkiewiczintegral} Let $h$ be a positive lower-semicontinuous function on $\R^n$ and $k$ an integer. Let $\mathcal W(k)$ be the Whitney decomposition of the level set $\Omega_k = \{ h(x)>2^k\}$, that is, $\Omega_k = \cup_{Q\in \mathcal W(k)} Q$. We define the function
\begin{equation}\label{marcinkiewicz-defintion}
M_{p,q} h (x) ^p = \sum_{k\in \Z} \sum_{Q\in \mathcal W(k)} 2^{kp} (M{\car Q}(x))^q.
\end{equation}
\end{definition}
We need lower-semicontinuity in this definition to ensure that we can apply Whitney's decomposition theorem. In the practice, we will apply this operator to $Mf$ and to $T^*f$, which are always lower-semicontinuous.

\begin{lemma}\label{Lemma4Sawyer} For a bounded, compactly supported function $ f$ and a weight $w\in C_q$ with $q>p$, we have
\begin{equation}\label{eq00-Sawyer} 
\int_{\R^n} (M_{p,q}M f(x))^p w(x)dx \leq \left(c_n 2^{c_n\frac{pq}{q-p} } \frac 1 {\eps} \log \frac1\eps\right)\int_{\R^n} ( M f(x))^pw(x)dx,
\end{equation}
where $M_{p,q}$ denotes the Marcinkiewicz integral operator as defined in \eqref{marcinkiewicz-defintion}.
\end{lemma}

\begin{proof}
Let $\mathcal W(k)$ be the Whitney decomposition of $\Omega_k = \{ Mf>2^k\} $, for any integer $k$. Let $N$ be a positive integer to be chosen later and fix a cube $P$ from the $k-N$ generation. We have, as in \cite{Sawyer1},
\begin{equation}\label{ClaimSawyer}
|\Omega_k\cap 5 P| \leq C 2^{- N} |P|,
\end{equation}
where $C$ depends only on the dimension $n$. 

 Now define the partial sums 
\[
 S(k) = 2^{kp}\sum_{Q\in \mathcal W(k)} \int_{\R^n} (M\car Q)^q w
 \]
and 
\[
S(k;N,P) = 2^{kp}\sum_{\substack{Q\in\mathcal W(k)\\ Q\cap P\neq \emptyset}}  \int_{\R^n} (M\chi_{\strt{1.5ex}Q})^q w,
\]
where in the last sum $P\in \mathcal W(k-N)$ is fixed. Because of the Whitney decomposition, $Q \cap R \neq \emptyset$ implies $Q \subset 5 P$ for large $N$, so we have 
\begin{align*}
S(k;N,P) & \leq \int_{\R^n} 2^{kp} \sum_{\substack{Q\in \mathcal W(k) \\ Q\subset 5 P}} (M\car R)^q w \\
& = \int_{10 P} + \int_{(10P)^c} \sum_{\substack{Q\in \mathcal W(k) \\ Q\subset 5 P}} (M\car R)^q w  = I + II \quad \text{for large } \ N.
\end{align*}
Now, by \eqref{eq14-Sawyer}, for any $\eta >0$, which  will be chosen  chosen later, and for $R=10$ we get
$$ I \leq 2^{kp}\frac 1{a\eps} \log \frac{c_n 10^{nq}}{\eta \eps} w(10P) + \eta 2^{kp} \int_{\R^n} (M\car P)^qw.$$
Standard estimates for the maximal function of characteristics of cubes show that if $x_P$ is the center of the cube $Q_P$ then
\begin{align*}
II & \leq c_{n}^q 2^{kp} \int_{(10P)^c} \frac{\sum |Q|^q}{|x-x_P|^{nq}} w(x)dx \\
& \leq c_{n}^q 2^{kp} \int_{(10P)^c} \frac{|\Omega_k \cap P|^q}{|x-P|^{nq}} w(x)dx \\
&\leq c_{n}^q 2^{kp} \int_{(10P)^c} \frac{2^{- {qN}}|P|^q}{|x-x_P|^{nq}} w(x)dx \\
& \leq c_{n}^q 2^{ N(p- q)+(k-N)p} \int_{\R^n} (M\car P)^q w,
\end{align*}
where we have used \eqref{ClaimSawyer} on the third inequality. Thus we have, by the Whitney decomposition theorem, for $N$ large,
%
\begin{align*}
S(k) & \leq \sum_{P\in \mathcal W(k-N)} S(k;N,P) \\
& \leq  \frac 1{a\eps} \log \frac{c_n 10^{nq}}{\eta \eps} 2^{kp} \int_{\R^n} \sum_{P\in \mathcal W(k-N)} \big(\car{10P}\big) w + ( \eta 2^{Np} + c_n^q 2^{N(p-q)}) S(k-N) \\
& \leq c_n \frac 1{a\eps} \log \frac{c_n 10^{nq}}{\eta \eps} 2^{kp} w(\Omega_{k-N}) + ( \eta 2^{Np} + c_n^q 2^{N(p-q)}) S(k-N)\\
&= c_n 2^{Np}\frac 1{a\eps} \log \frac{c_n 10^{nq}}{\eta \eps} 2^{p(k-N)} w(\Omega_{k-N}) +( \eta 2^{Np} + c_n^q 2^{N(p-q)}) S(k-N).
\end{align*}
Now, since $q>p$, we can chose $N$ so that $c_{n,q} 2^{N(p-q)}< \frac14$, that is, $N\geq c_n \frac{q}{q-p}$; and $\eta$ so that $\eta 2^{Np}<\frac 14$.
%
%
\begin{align*}
S(k) & \leq c_{n} 2^{c_n \frac{pq}{q-p}}\frac1{a\eps} (qc_n + \log \frac{1}{\eps}+ c_n \frac{pq}{q-p}) 2^{p(k-N)}w(\Omega_{k-N}) + \frac 12 S(k-N)\\
& \leq c_n 2^{c_n \frac{qp}{q-p}} \frac 1{\eps}\log \frac1\eps 2^{p(k-N)} w(\Omega_{k-N}) + \frac 12 S(k-N).
\end{align*}Thus, with $S_M = \sum_{k\leq M }S(k)$ we get
$$S_M \leq \frac12 S_M +  c_n 2^{c_n \frac{qp}{q-p}} \frac 1{\eps}\log \frac1\eps\int_{\R^n} (Mf)^pw.$$

Now, exactly as in \cite{Sawyer1}, p. 260,  we have that $S_M<\infty$ and since it is clear that 
$$\sup_M S_M = \int _{\R^n}(M_{p,q}(M f))^p w,$$
we conclude the proof of the lemma.
\end{proof}

\begin{remark}
The important part of the dependence of the constant on the exponents $p$ and $q$ is that the lemma will fail to be true for $p=q$, with this kind of blowup.
\end{remark}

\begin{lemma}\label{Lemma2Sawyer} Under the same assumptions of Theorem \ref{norminequalitythm} we have
\begin{align*}
\int_{\R^n} (M_{p,q} T^* f(x))^p w(x)dx &\leq  \left( c_n 2^p\frac1{a\eps}\log \frac{c_n 10^{nq}2^{p+2}}{\eps} \right) \int_{\R^n} (T^*f(x))^pw(x)dx \\ 
&+\left(c_n^q2^{c_n\frac{p^2q}{q-p} } \frac 1 {\eps^2} \log \frac1\eps\right) \int_{\R^n} ( M f(x))^pw(x)dx.
\end{align*}
\end{lemma}

\begin{proof}
Let $\mathcal W(k)$ be the Whitney decomposition of the level set $\Omega_k = \{ x\in \R^n:  T^*f(x)>2^k\} $ for integer $k$. One can prove as in \cite{CoifmanFeffermanpaper} the following inequality:  if $Q\in \mathcal W(k-1)$ and  $5Q\not \subset \{ Mf >2^{k-N}\}$ for some $N\geq1$, then 
\begin{equation} \label{goodlambdamala} 
|\{x \in Q ; T^*f >2^k\}| \leq C_T 2^{-N}|Q|.
\end{equation}

Let $\mathcal V(k)$ be the Whitney decomposition of the set $\{ Mf >2^k\}$. We observe that for each cube $Q\in \mathcal W(k-1)$ there are two cases (for a fixed $N$ that we will chose later). 

Case $(a)$.  $5Q\subset \{ Mf>2^{k-N}\}$ in which case $5Q \subset c_n I$ for some $I\in \mathcal V(k-N)$.

Case $(b)$. $5Q\not \subset \{ Mf>2^{k-N}\}$ in which case \eqref{goodlambdamala} implies $$\sum_{\substack{ P \in \mathcal W(k)\\ P \subset 5 Q}} |P| \leq c_T 2^{-N} |Q|.$$

Now define the partial sums 
\[
S(k) = \sum_{Q\in \mathcal W(k)} 2^{kp}\int_{\R^n} (M\car Q)^q w
\]
and
\[
S(k;P) = \sum_{\substack{Q\in \mathcal W(k) \\ Q\cap P \neq \emptyset}} 2^{kp}\int_{\R^n} (M\car Q)^q w 
\leq \sum_{\substack{Q\in \mathcal W(k) \\ Q\subset 5 P} }2^{kp} \int_{\R^n} (M\car Q)^q w.
\]
Here, $P\in \mathcal W(k-1)$ and the last inequality follows from the Whitney decomposition. Thus,
\[
S(k;P) \leq \sum_{\substack{Q\in \mathcal W(k) \\ Q\subset 5 P} }2^{kp} \int_{\R^n} (M\car Q)^q w = \int_{10P }+ \int_{(10P)^c } \sum_{\substack{Q\in \mathcal W(k) \\ Q\subset 5 P} }2^{kp} (M\car Q)^q w = I + II.
\]
By \eqref{eq14-Sawyer} with $R=10$ we have 
\[
 I \leq c_{n} \frac 1{a\eps} \log \frac {c_n 10^{nq}}{\eps \eta} 2^{kp} w(5P) + \eta 2^{kp} \int_{\R^n} (M\car P)^qw,
 \]
where $\eta>0$ is a positive number at our disposal. As in the previous lemma one can show
$$II \leq c^q_n 2^{kp-Nq} \int_{\R^n} (M\car P)^qw.$$
Combining estimates for $I $ and $II $ we obtain, for every case $(b)$ cube $P \in \mathcal W(k-1)$,
\begin{equation}\label{case2}
S(k;P) \leq c_n \frac1{a\eps}\log \frac{c_n 10^{nq}}{\eps \eta} 2^{kp}  w(5P) +  (\eta + c_n^q 2^{-Nq})2^{kp}\int_{\R^n} (M\car P)^qw.
\end{equation}
Thus
\[
S(k) \leq \sum_{\substack{P\in \mathcal W(k-1) \\ P  \text{ is }  (a)}} S(k;P) + \sum_{\substack{P\in \mathcal W(k-1) \\ P  \text{ is }  (b)}} S(k;P) = III + IV.
\]
Now, since each of the $Q\in \mathcal W(k)$ of type $(a)$ intersects at most $c$ of the $P\in \mathcal W(k-1)$, (yet again due to the Whitney decomposition), we have
\[ \Rmnum{3} \leq c \sum_{I\in \mathcal V(k-N)} \sum_{\substack{Q\in \mathcal W(k)\\ Q \subset c_n I}} 2^{ kp} \int_{\R^n} (M\car Q)^qw \leq c_n^q \frac{1}{\eps} \sum_{I\in \mathcal V(k-N)} 2^{kp} \int_{\R^n} (M\car I)^qw,
\]
where we have used \eqref{eq15-Sawyer} and $M\car{c_nI} \leq c_n M\car I$ (for two different $c_n$ of course).
For the remaining part we have by \eqref{case2} 
\begin{align*}
IV &\leq c_n \frac1{a\eps}\log \frac{c_n 10^{nq}}{\eps \eta} 2^{kp}  \int_{\R^n} w(\Omega_{k-1}) +  (\eta2^p + c_n^q 2^{p-Nq})2^{(k-1)p} \sum_{P\in \mathcal W(k-1)} \int_{\R^n} (M\car P)^qw \\
& \leq c_n 2^p\frac1{a\eps}\log \frac{c_n 10^{nq}}{\eps \eta} 2^{(k-1)p} w( \Omega_{k-1}) + \frac12 S(k-1),
\end{align*}
if we choose $\eta$ small enough and $N$ big enough. This means $\eta = 2^{-(p+2)}$ and $N \geq c_n \frac{p+q}{q}$. Combining now estimates for $III$ and $IV$ we get
\begin{align*}
S(k) \leq \frac12 S(k-1) &+ \left( c_n 2^p\frac1{a\eps}\log \frac{c_n 10^{nq}2^{p+2}}{\eps} \right) 2^{(k-1)p} w( \Omega_{k-1}) \\
&+ \left(c_n^q2^{c_n \frac pq(p+q)} \frac{1}{\eps} \right) \sum_{I\in \mathcal V(k-N)} 2^{(k-N)p} \int_{\R^n} (M\car I)^qw.
\end{align*}
Set $S_M = \sum_{k\leq M } S(k) $ and sum the previous inequality over $k\leq M$ to obtain
\begin{align*}
S_M & \leq \frac12 S_M + \left( c_n 2^p\frac1{a\eps}\log \frac{c_n 10^{nq}2^{p+2}}{\eps} \right) \int_{\R^n} (T^*f)^pw \\
& \quad +\left(c_n^q2^{c_n \frac pq(p+q)} \frac{1}{\eps} \right) \int_{\R^n} (M_{p,q}(Mf))^pw  \\
& \leq \frac 12 S_M + \left( c_n 2^p\frac1{a\eps}\log \frac{c_n 10^{nq}2^{p+2}}{\eps} \right) \int_{\R^n} (T^*f)^pw \\
& \quad  + \left(c_n^q2^{c_n \frac pq(p+q)} \frac{1}{\eps} \right) \left(c_n 2^{c_n\frac{pq}{q-p} } \frac 1 {\eps} \log \frac1\eps\right)  \int_{\R^n}  (Mf)^pw,
\end{align*} 
by \eqref{eq00-Sawyer}. It can be shown (cf. \cite{Sawyer1}, p.262) that $S_M<\infty$, so taking it to the left and then taking the supremum over all $M$ we obtain the desired result.
\end{proof}

\begin{proof}[Proof of theorem \ref{norminequalitythm}]
Using the exponential decay from \cite{Buckley1993}, we know that if we write $\{T^*f>2^k\} = \cup_jQ_j$ as in the Whitney decomposition theorem, we have
\begin{equation}\label{goodlambdabuena}
|\{x\in Q_j : T^*f(x)>2\la,  Mf(x) \leq \gamma \la\}| \leq c e^{-\frac c\gamma} |Q_j|,
\end{equation}
for any $\gamma>0$. We call $E_j$ to the set in the left side of \eqref{goodlambdabuena}. Then, if we call $r$ to the exponent $1+\delta$ in Theorem \ref{RHI-Cp-Theorem}, we get
\begin{align*}
w(E_j) &= |E_j| \frac 1{|E_j|} \int_{E_j} w \leq |E_j| \left( \frac 1{|E_j|} \int_{E_j} w^r \right) ^\frac1r \leq |E_j|^{\frac1{r'}} |Q_j|^\frac 1r \left(\frac 1{|Q_j|} \int_{Q_j} w^r \right) ^\frac1r \\
& \leq  |E_j|^{\frac1{r'}} |Q_j|^\frac 1r \frac 2{|Q_j|} \int_{\R^n} (M\chi_{\strt{1.5ex}Q_j})^q w  \leq c e^{-\frac c{\gamma r'}} \int_{\R^n}  (M\chi_{\strt{1.5ex}Q_j})^q w.
\end{align*}

We use the standard good-$\la$ techniques as in \cite{Sawyer1} combined with Lemma \ref{Lemma2Sawyer} to get
\begin{align*}
\int_{\R^n} (T^*f)^p w & \leq \left( \frac2\gamma\right) ^p \int_{\R^n} (Mf)^pw + c e^{-\frac c{\gamma r'}} \int_{\R^n} (M_{p,q}T^*f)^pw \\
&  \leq \left( 2^p \gamma^{-p} + e^{-\frac c{\gamma r'}} \left(c_n^q2^{c_n\frac{p^2q}{q-p} } \frac 1 {\eps^2} \log \frac1\eps\right) \right) \int_{\R^n} (Mf)^pw \\
& \quad+ c e^{-\frac c{\gamma r'}} \left( c_n 2^p\frac1{a\eps}\log \frac{c_n 10^{nq}2^{p+2}}{\eps} \right) \int_{\R^n}( T^*f)^pw 
\end{align*}
Choosing $\gamma^{-1} \sim c_n (q+\frac{p^2q}{q-p})\frac 1\eps \log \frac 1 \eps$ we can make $$  e^{-\frac c{\gamma r'}} \left(c_n^q2^{c_n\frac{p^2q}{q-p} } \frac 1 {\eps^2} \log \frac1\eps\right)  < \frac 12$$
and 
$$ c e^{-\frac c{\gamma r'}} \left( c_n 2^p\frac1{a\eps}\log \frac{c_n 10^{nq}2^{p+2}}{\eps} \right) <\frac 12$$
and taking the term to the left side (which is possible since it is finite, see \cite{Sawyer1}) we obtain
\begin{equation*}
\int_{\R^n} (T^*f)^p w \leq c_n^p \left( c_n (q+\frac{p^2q}{q-p})\frac 1\eps \log \frac 1 \eps \right)^p \int_{\R^n} ( Mf)^pw. \qedhere
\end{equation*}
\end{proof}

\begin{remark}
We conjecture that the first $q$ in the constant should not be there. That way $\lim_{q\rightarrow \infty} c_q <\infty.$ We think this should be the case because whenever $w\in C_q$ and $q$ is bigger, we have more information. This way we could recover a weighted inequality for the $A_\infty$ class, though it would be a worse one than the one we mention in the introduction.
For this very reason, we conjecture that the dependence on the $C_q$ constant is not sharp in this sense.
\end{remark}

\end{document}